\documentclass[a4paper,10pt]{amsart}

\usepackage[english]{babel}
\usepackage[utf8]{inputenc}
\usepackage{float}
\usepackage{amsmath}
\usepackage{amsfonts}
\usepackage{amssymb}
\usepackage{verbatim}
\usepackage{dsfont}
\usepackage{setspace}
\usepackage{fullpage}

\usepackage{amsthm}

\newenvironment{dem}{\noindent{\bf Proof.}}{\qed\\}

\newtheorem{defi}{Definition}[section]
\newtheorem{theo}[defi]{Theorem}
\newtheorem{prop}[defi]{Proposition}
\newtheorem{lem}[defi]{Lemme}

\newtheorem{rmq}[defi]{Remark}

\numberwithin{equation}{section}

\begin{document}
\title{An $\ell_1$-oracle inequality for the Lasso in finite mixture of multivariate Gaussian regression models}
\author{Emilie Devijver}\address{Laboratoire de Mathématiques d'Orsay, Faculté des Sciences d'Orsay, Université Paris-Sud, 91405 Orsay, France ; e-mail: emilie.devijver@math.u-psud.fr}
%
%
\begin{abstract}
We consider a multivariate finite mixture of Gaussian regression models for high-dimensional data, where the number of covariates and the size of the response may be much larger than the sample size.
We provide an $\ell_1$-oracle inequality satisfied by the Lasso estimator according to the Kullback-Leibler loss.
This result is an  extension of the $\ell_1$-oracle inequality established by Meynet in \cite{Meynet} in the multivariate case.
We focus on the Lasso for its $\ell_1$-regularization properties rather than for the variable selection procedure, as it was done in St\"adler in \cite{Stadler}.
\end{abstract}

\maketitle
\section*{Introduction}
Finite mixture regression models are useful for modeling the relationship between response and predictors, arising from different subpopulations.
Due to computer progress, we are faced with high-dimensional data where the number of variables can be much larger than the sample size.
Moreover, the response variable could be high-dimensional.
We have to reduce the dimension to avoid identifiability problems.
Considering a mixture of linear models, an assumption widely used is to say that only a few covariates explain the response.
Among various methods, we focus on the $\ell_1$-penalized least squares estimator of parameters to lead to sparse regression matrix.
Indeed, it is a convex surrogate for the non-convex $\ell_0$-penalization, and produce sparse solutions.
First introduced by Tibshirani in \cite{Tibshirani}, the Lasso estimator is defined by
$$\hat{\beta}(\lambda) = \underset{\beta \in \mathbb{R}^{p}}{\operatorname{argmin}} \left\{ ||Y-X\beta ||_2^2 + \lambda ||\beta||_1 \right\}, \hspace{1cm} \lambda>0,$$
in a linear model, where $X \in \mathbb{R}^p$ and $Y \in \mathbb{R}$.
Many results have been proved to study the performance of this estimator.
For example, cite \cite{BickelRitovTsybakov,EfronHastie,Hebiri} for studying this estimator as a variable selection procedure in this linear model case.
Note that those results need strong restrictive eigenvalue assumptions on the Gram matrix $X^t X$, that can be not fulfilled in practice.
A summary of assumptions and results is given by B\"uhlmann and van de Geer in \cite{BuhlmannVDG}. One can also cite van de Geer in \cite{VDG} and discussions, who precise a chaining argument to perform rate, even in a non linear case.

If we assume that data arise from different subpopulations, we could work with finite mixture regression models.
The homogeneity assumption of the linear model is often inadequate and restrictive.
With this model, one can mention St\"adler et al., in \cite{Stadler}, who work with finite mixture of linear models.
They assume that, for $i= 1,\ldots,n$, $Y_i$ follows a law density $s_\psi(.|x_i)$ which is a finite mixture of $k$ Gaussian densities with proportion vector $\pi$,
$$Y_i|X_i=x_i \sim s_\psi(Y_i|x_i) = \sum_{r=1}^k \frac{\pi_r}{\sqrt{2 \pi}\sigma_r} \exp\left( -\frac{(Y_i-\mu_r^t x_i)^2}{2 \sigma_r^2} \right)$$
for some parameters $\psi = (\pi_r,\mu_{r,j},\sigma_r)_{r \in \{1,\ldots,k\} ; j \in \{1,\ldots,p\}}$.
They extend the Lasso estimator by
$$\hat{s}(\lambda) = \underset{s_\psi}{\operatorname{argmin}} \left\{ -\frac{1}{n} \sum_{i=1}^n \log(s_\psi(Y_i|x_i)) + \lambda \sum_{r=1}^k \sum_{j=1}^p |\mu_{r,j}| \right\}, \hspace{1cm} \lambda>0$$

For this estimator, they provide an $\ell_0$-oracle inequality satisfied by this Lasso estimator, according to restricted eigenvalue conditions also, and margin conditions, which lead to link the Kullback-Leibler loss function to the $\ell_2$-norm of the parameters.

Another way to study this estimator is to look after the Lasso for its $\ell_1$-regularization properties.
For example, cite \cite{MassartMeynet,Meynet,RigolletTsybakov}. Contrary to the $\ell_0$-results, some $\ell_1$-results are valid with no assumptions,
neither on the Gram matrix, nor on the margin. This can be achieved due to the fact that we are only lookinf for rate of convergence of order $\frac{||s_\psi|_1}{\sqrt{n}}$ rather than $\frac{||s_\psi|_0}{n}$.
For instance, we could cite Meynet in \cite{Meynet} who give an $\ell_1$-oracle inequality for the Lasso estimator
for the finite mixture Gaussian regression models.
In this paper, we extend this result to finite mixture of multivariate Gaussian regression models.
Indeed, we will work with $(X,Y) \in \mathbb{R}^p \times \mathbb{R}^q$.
As in \cite{Meynet}, we shall restrict to the fixed design case, that is to say non-random regressors.
Under only bounded parameters assumption, we provide a lower bound on the Lasso regularization parameter $\lambda$, to guarantee such an oracle inequality.

This result is non-asymptotic: the number of observations is fixed, and the number $p$ of covariates can grow.
The number $k$ of clusters in the mixture is fixed.
It is deduced from a finite mixture Gaussian regression model selection theorem for $\ell_1$-penalized maximum likelihood condition density estimation.
We establish it following the one of Meynet in \cite{Meynet}, which combines Vapnik's structural risk minimization method (see Vapnik in \cite{Vapnik}) and theory around model selection 
(see Le Pennec and Cohen in \cite{Lepennec} and Massart in \cite{Massart}).
As in Massart and Meynet in \cite{MassartMeynet}, our oracle inequality is deduced from this general theorem because the Lasso is viewed as the solution
of a penalized maximum likelihood model selection procedure over a countable collection of $\ell_1$-ball models.

The article is organized as follows. The notations and the framework are introduced in Section \ref{framework}.
In Section \ref{SectionInegOracle}, we state the main result of the article, which is an $\ell_1$-oracle inequality satisfied by the Lasso in finite mixture of multivariate Gaussian regression models.
Section \ref{proof} is devoted to the proof of this result, deriving from two easier propositions. Those propositions are proved in Section \ref{proofProposition}, whereas details of lemma states in Section \ref{proofLemme}.

\section{Notations and framework}
\label{framework}
\subsection{Finite mixture regression model}
We observe $n$ independent couples $((x_{i}, y_{i}))_{1 \leq i \leq n}$ of random variables $(X,Y)$, with $Y_i  \in \mathbb{R}^q$ and $X_i \in \mathbb{R}^p$, coming from a probability distribution with unknown conditional density $s_0$.
We assume that the model could be estimated by a finite mixture of $k$ Gaussian regressions. Then we assume that the data come from several subpopulations, each of them following a conditional density estimated by a multidimensional Gaussian density.

The random response variable $ Y \in \mathbb{R}^q$ depends on a set of explanatory variables, written $X \in \mathbb{R}^p$, through a regression-type model.
We assume that:
\begin{itemize}
 \item the variables $Y_i|X_i$ are independent, for all $i=1,\ldots, n$ ;
 \item the variables $Y_i|X_i=x_i \sim s_{\xi}(y|x_i)dy$, with
  \begin{align}
  &s_{\xi}(y|x_i)=\sum_{r=1}^{k} \frac{\pi_{r}}{(2 \pi)^{\frac{q}{2}}  \text{det}(\Sigma_r)^{1/2}} \exp \left( -\frac{(y-\beta_{r} x_i)^t \Sigma_{r}^{-1}(y-\beta_{r} x_i)}{2} \right)\\
  &\xi=(\pi_{1},\ldots, \pi_{k},\beta_{1},\ldots,\beta_{k},\Sigma_{1},\ldots,\Sigma_{k}) \in \Xi_k = \left( \Pi_{k} \times (\mathbb{R}^{q\times p})^k\times (\mathbb{S}^q_+)^k \right) \nonumber\\
  & \Pi_{k} = \left\{ (\pi_{1}, \ldots, \pi_{k}) ; \pi_{r} >0 \text{ for } r \in \{1,\ldots, k\} \text{ and } \sum_{r=1}^{k} \pi_{r} = 1 \right\} \nonumber\\
  & \mathbb{S}_+^q \text{ is the set of symmetric positive definite matrices on } \mathbb{R}^q . \nonumber
 \end{align}
\end{itemize}

We want to estimate the conditional density function $s_\xi$ from the observations. 
For all $r \in \{1,\ldots, k\}, \beta_{r}$ is the matrix of regression coefficients, and $\Sigma_{r}$ is the covariance matrix in the mixture component $r$. 
The $\pi_{r}$s are the mixture proportions.
In fact, for all $r \in \{1,\ldots,k\}$, for all $z \in \{1,\ldots,q\}$, $\beta_{r}^t x= \sum_{j=1}^{p} \beta_{r,j,.} x_{j}$ is the mean coefficient of the mixture component $r$ for the conditional density $s_{\xi}(.|x)$.

\subsection{Boundedness assumption on the mixture and component parameters}
Let, for a matrix $A$, $||A||_{\min}$ the smallest coefficient of $A$, and $||A||_{\max}$ the greatest coefficient of $A$.
We shall restrict our study to bounded parameters vector $\xi =( \beta, \Sigma,\pi) \in \Xi_k$. 
Specifically, we assume that there exists
deterministic positive constants $a_\beta, A_\beta, a_\Sigma, \tilde{A}_\Sigma,\tilde{a}_\Sigma, A_\Sigma, a_\pi$  such that $\xi$ belongs to $\tilde{\Xi}_k$, with

\begin{multline}
\label{paraborn}
\tilde{\Xi}_k = \left\lbrace  \xi \in \Xi_k: \text{ for all } r \in \{1,\ldots,k \}, a_\beta \leq \min_{z \in \{1 \ldots q\}} \inf_{x \in \mathbb{R}^p}  || \beta_{r,z} x || \leq \max_{z \in \{1 \ldots q\} } \inf_{x \in \mathbb{R}^p}  || \beta_{r,z} x || \leq A_\beta, \right. \\
 \left. \tilde{a}_\Sigma \leq ||\Sigma_{r}||_{\min} \leq ||\Sigma_{r}||_{\max} \leq \tilde{A}_\Sigma,a_\Sigma \leq ||\Sigma_{r}^{-1}||_{\min} \leq ||\Sigma_{r}^{-1}||_{\max} \leq A_\Sigma, a_\pi \leq \pi_r \right\rbrace.
\end{multline}

Let $S$ the set of conditional densities $s_\xi$ in this model:
$$S= \left\lbrace s_\xi, \xi \in \tilde{\Xi}_k \right\rbrace.$$
To simplify the proofs, we also assume that $s_0$ belongs to  $S$: there exists $\xi_0=(\beta_0,\Sigma_0,\pi_0) \in \tilde{\Xi}_k$ such that $s_0 =s_{\xi_0}$.

\subsection{Maximum likelihood estimator and penalization}
In a maximum likelihood approach, the loss function taken into consideration is the KL information, which is defined for two densities $s$ and $t$ by 
$$KL(s,t)=\int_{\mathbb{R}} \log \left( \frac{s(y)}{t(y)}\right) s(y) dy.$$

In a regression framework, we have to adapt this definition to take into account the structure of conditional densities.
For fixed covariates $(x_1,\ldots,x_n)$, we consider
$$KL_n(s,t)=\frac{1}{n} \sum_{i=1}^n KL(s(.|x_i),t(.|x_i)) = \frac{1}{n} \sum_{i=1}^n \int_{\mathbb{R}} \log \left( \frac{s(y|x_i)}{t(y|x_i)} \right) s(y|x_i) dy.$$
Using the maximum likelihood approach, we will estimate $s_0$ by the conditional density $s_{\xi}$ which maximize the likelihood conditionally to $(x_i)_{1\leq i\leq n}$.
Nevertheless, we work with high-dimensional data, then we have to regularize the maximum likelihood estimator to reduce the dimension.
We consider the $\ell_1$-regularization
$$\hat{s}(\lambda) := \underset{s_\xi \in  S}{\operatorname{argmin}} \left\{ -\frac{1}{n} \sum_{i=1}^n \log(s_\xi(Y_i|x_i)) + \lambda |s_{\xi}|_1 \right\} ;$$
where $\lambda >0$ is a regularization parameter, and 
$$|s_\xi|_1 = \sum_{r=1}^k \sum_{z=1}^q \sum_{j=1}^p |\beta_{r,j,z}|$$
for $\xi = (\beta_r,\Sigma_r,\pi_r)_{r=1,\ldots,k}$.

\section{Oracle inequality}
\label{SectionInegOracle}
In this section, we provide an $\ell_1$-oracle inequality satisfied by the Lasso estimator in finite mixture multivariate Gaussian regression models.

We define $||x||_{\max,n} =\sqrt{ \frac{1}{n} \sum_{i=1}^{n} \max_{j=1,\ldots,p} |x_{i,j}|^2 }$.
\begin{theo}
\label{inegOracle}
 Assume that 
$$\lambda \geq \kappa\left( A_\Sigma \vee \frac{1}{a_\pi}\right) \left( 1+4(q+1) A_\Sigma \left(A_\beta^2+\frac{\log(n)}{a_\Sigma}\right) \right) \sqrt{\frac{k}{n}} \left( 1 + q ||x||_{\text{max},n} \log(n) \sqrt{k \log(2p+1)} \right)$$  
with $\kappa$ an absolute positive constant.
 Then, the Lasso estimator, denoted by $\hat{s}(\lambda)$, defined by
 $$ \hat{s}(\lambda) = \underset{s_\xi \in  S}{\operatorname{argmin}} \left( - \frac{1}{n} \sum_{i=1}^n \log(s_\xi (Y_i | X_i)) + \lambda |s_\xi|_1 \right) ;$$
satisfies the $\ell_1$-oracle inequality.

\begin{align*}
&E[KL_n(s_0,\hat{s}(\lambda)) ]\\
&\leq (1+\kappa^{-1}) \inf_{s_\psi \in S} \left(  KL_n (s_0,s_\psi)+\lambda |s_\psi|_1 \right) + \lambda \\
& + \sqrt{\frac{k}{n}} \kappa^{'} \left[ \frac{e^{-\frac{1}{2}-\frac{1}{4} a_{\beta}^2 a_\Sigma} \pi^{q/2} a_\pi}{(qA_\Sigma)^{q/2}} \sqrt{2q}  \right. \\
&+ \left. \left( A_\Sigma \vee \frac{1}{a_\pi} \right) \left( 1+4(q+1) A_\Sigma\left(A_\beta^2+\frac{\log(n)}{a_\Sigma} \right)\right) k\left(1+A_\beta+\tilde{A}_\Sigma\right)^2 \right]
\end{align*}

where $\kappa'$ is a positive constant.
\end{theo}

This theorem provides information about the performance of the Lasso as an $\ell_1$-regularization algorithm. 
If the regularization parameter $\lambda$ is properly chosen, the Lasso estimator, which is the solution of the $\ell_1$-penalized empirical risk minimization problem, 
behaves as well as the deterministic Lasso, which is the solution of the $\ell_1$-penalized true risk minimization problem, up to an error term of order $\lambda$.

Our result is non-asymptotic: the number $n$ of observations is fixed while the number $p$ of covariates can grow with respect to $n$ and can be much larger than $n$. The numbers $k$ of clusters in the mixture is fixed.

There is no assumption neither on the Gram matrix, nor on the margin, which are classical assumptions for oracle inequality, as done in \cite{Stadler}.
Moreover, this kind of assumptions involve unknown constants, whereas here, every constants are explicit.

\begin{rmq}
 Van de Geer, in \cite{VDG}, gives some tools to improve the bound of the regularization parameter to $\sqrt{\frac{\log(p)}{n}}$.
 Nevertheless, we have to control eigenvalues of the Gram matrix of some functions $(\psi_j(X_i))_{j=1,\ldots,D, i=1,\ldots,n}$, $D$ being the number of parameters to estimate, where $\psi_j(X_i)$ satisfies
$$| \log(s_\xi(Y_i[X_i))-\log (s_{\tilde{\xi}}(Y_i|X_i))|\leq \sum_{j=1}^D |\xi_j - \tilde{\xi}_j| \psi_j(X_i).$$
 
 In our case of mixture of regression models, control eigenvalues of the Gram matrix of $(\psi_j(X_i))$ correspond to make some assumptions, as REC, to avoid dimension reliance on $n,k$ and $p$.
 Without this king of assumptions, we could not guarantee that our bound is in order of $\sqrt{\frac{\log(p)}{n}}$, because we could not guarantee that eigenvalues does not depend on dimensions.
 In order to get a result, with smaller assumptions, we do not use the chaining argument developed in \cite{VDG}
.
Nevertheless, one can easily compute that, under restricted eigenvalue condition, we could perform the order of the regularization parameter do $\lambda \asymp \sqrt{\frac{\log(p)}{n}} \log(n)$.
\end{rmq}

\section{Proof of the oracle inequality}
\label{proof}
\subsection{Main propositions used in this proof}
The first result we will prove is the next theorem, which is an $\ell_1$-ball mixture multivariate regression model selection theorem for $\ell_1$-penalized maximum likelihood conditional density estimation in the Gaussian framework.
\begin{theo}
\label{theo interm}
We observe $((X_i,Y_i))_{i=1,\ldots,n}$ with unknown conditional Gaussian mixture density $s_0$.
For all $m \in \mathbb{N}^*$, we consider the $\ell_1$-ball  $S_{m}=\{s_\xi \in S, |s_\xi|_1 \leq m \}$ and $\hat{s}_{m}$ a $\eta_{m}$-log-likelihood minimizer in $S_{m}$, for $\eta_{m} \geq 0$:
$$- \frac{1}{n} \sum_{i=1}^n \log (\hat{s}_{m} (Y_i|X_i)) \leq \inf_{s_{m} \in S_{m}} \left( - \frac{1}{n} \sum_{i=1}^n \log(s_m(Y_i|X_i)) \right) + \eta_m.$$
Assume that for all $m \in \mathbb{N}^*$, the penalty function satisfies $\text{pen}(m)= \lambda m$ with 
$$\lambda \geq \kappa\left( A_\Sigma \vee \frac{1}{a_\pi}\right)  \left( 1+4(q+1) A_\Sigma\left(A_\beta^2+\frac{\log(n)}{a_\Sigma} \right)\right) \sqrt{\frac{k}{n}} \left( 1 + q ||x||_{\text{max},n} \log(n) \sqrt{k \log(2p+1)} \right)$$
for a constant $\kappa$. Then, for all estimator $\hat{s}_{\hat{m}}$ with $\hat{m}$ such that 
$$- \frac{1}{n} \sum_{i=1}^n \log(\hat{s}_{\hat{m}}(Y_i|X_i)) + \text{pen} (\hat{m}) \leq \inf_{m \in \mathbb{N}^*} \left( - \frac{1}{n} \sum_{i=1}^n \log(\hat{s}_{m} (Y_i|X_i))+pen(m)\right) + \eta$$
for $\eta \geq 0$, it satisfies
\begin{align*}
&E(KL_n(s_0,\hat{s}_{\hat{m}}))\leq (1+\kappa^{-1}) \inf_{m\in\mathbb{N}^*} \left( \inf_{s_m \in S_m} KL_n (s_0,s_m)+ \text{pen}(m) + \eta_m \right) +\eta \\
& + \kappa^{'} \sqrt{\frac{k}{n}} \frac{e^{-\frac{1}{2}-\frac{1}{4} a_{\beta}^2 a_\Sigma} \pi^{q/2}}{(qA_\Sigma)^{q/2}} \sqrt{2q a_\pi} \\
&+ \kappa^{'} \sqrt{\frac{k}{n}} \left[ \kappa' k \left( A_\Sigma \vee \frac{1}{a_\pi} \right) \left( 1+\frac{4(q+1)}{2} A_\Sigma\left(A_\beta^2+\frac{\log(n)}{a_\sigma}\right) \right) (1+A_\beta+\tilde{A}_\Sigma)^2 \right]
\end{align*}

where $\kappa^{'}$ is a positive constant.
\end{theo}

It is an $\ell_1$-ball mixture regression model selection theorem for $\ell_1$-penalized maximum likelihood conditional density estimation in the Gaussian framework.
Its proof could be deduced from the two following propositions, which split the result according if the variable $Y$ is large enough or not.

\begin{prop}
\label{prop1}
We observe $((X_i,Y_i))_{i=1,\ldots,n}$, with conditional density unknown denoted by $s_0$. Let $M_n >0$, and consider the event
$$T:= \left\lbrace  \max_{i \in \{1,\ldots,n\}} \max_{z\in \{1,\ldots,q\}}  |Y_{i,z}| \leq M_n \right\rbrace.$$
For all $m \in \mathbb{N}^*$, we consider the $\ell_1$-ball  $$S_m = \{s_\xi \in S, |s_\xi|_1 \leq m\}$$
and $\hat{s}_m$ an $\eta_m$-log-likelihood minimizer in $S_m$, for $\eta_m \geq 0$:
$$- \frac{1}{n} \sum_{i=1}^n \log (\hat{s}_m (Y_i|X_i)) \leq \inf_{s_m \in S_m} \left( - \frac{1}{n} \sum_{i=1}^n \log(s_m(Y_i|X_i)) \right) + \eta_m.$$

Let $C_{M_n} = \max \left( \frac{1}{a_\pi}, A_\Sigma+\frac{1}{2} (|M_n|+A_\beta)^2 A_\Sigma^2, \frac{q(|M_n|+A_\beta)A_\Sigma}{2} \right)$.
Assume that for all $m \in \mathbb{N}^*$, the penalty function satisfies $\text{pen}(m)= \lambda m$ with 
$$\lambda \geq \kappa  \frac{4 C_{M_n}}{\sqrt{n}} \sqrt{k} \left(1+9q||x||_{\max,n} \log(n) \sqrt{k \log(2p+1)} \right)$$

for some absolute constant $\kappa$. Then, any estimate $\hat{s}_{\hat{m}}$ with $\hat{m}$ such that 
$$- \frac{1}{n} \sum_{i=1}^n \log(\hat{s}_{\hat{m}}(Y_i|X_i)) + \text{pen} (\hat{m}) \leq \inf_{m \in \mathbb{N}^*} \left( - \frac{1}{n} \sum_{i=1}^n \log(\hat{s}_{m} (Y_i|X_i))+\text{pen}(m)\right) + \eta$$
for $\eta \geq 0$, satisfies

\begin{align*}
 E(KL_n(s_0, \hat{s}_{\hat{m}}) \mathds{1}_\mathcal{T}) &\leq (1+\kappa^{-1}) \inf_{m \in \mathbb{N}^*} \left( \inf_{s_m \in S_m} KL_n(s_0,s_m) +\text{pen}(m) + \eta_m \right) \\
 &+ \frac{\kappa^{'} k^{3/2} q C_{M_n}}{\sqrt{n}} (1+(A_\beta+\tilde{A}_\Sigma)^2)
\end{align*}

where $\kappa^{'}$ is an absolute positive constant.
\end{prop}

\begin{prop}
\label{prop2}
Let $s_0,\mathcal{T}$ and $\hat{s}_{\hat{m}}$ defined as in the previous proposition. We assume that the conditional density $s_0$ of $((X_i,Y_i))_{i=1, \ldots,n}$ is a mixture of Gaussian conditional densities. Then,
$$E(KL_n (s_0,\hat{s}_{\hat{m}}) \mathds{1}_{T^c}) \leq \frac{e^{-1/2} \pi^{q/2}}{(q A_\Sigma)^{q/2}} \sqrt{2knqa_\pi} e^{-1/4 (M_n^2 -2M_nA_\beta+a_\beta^2)a_\Sigma}.$$
\end{prop}

\subsection{Notations}
\label{notations}
To prove those two propositions, and the theorem, begin with some notations.

For any measurable function $g: \mathbb{R}^q \mapsto \mathbb{R}$ , we consider the empirical norm
$$g_n:= \sqrt{\frac{1}{n}\sum_{i=1}^n g^2(Y_i|X_i)} ;$$
its conditional expectation
$$E_X(g)=E(g(.|X)|X=x) = \int_{\mathbb{R}^q} g(y|x)s_0(y|x)dy ;$$
its empirical processus
$$P_n(g):=\frac{1}{n} \sum_{i=1}^n g(Y_i|X_i) ;$$
and its normalized processus 
$$\nu_n(g):= P_n(g)-E_X(P_n(g)) = \frac{1}{n} \sum_{i=1}^n \left[ g(Y_i | X_i) - \int_{\mathbb{R}^q} g(y|x_i) s_0(y|x_i) dy \right].$$

For all $m \in \mathbb{N}^*$, for all model $S_m$, we define $$F_m=\left\lbrace f_m = - \log \left(\frac{s_m}{s_0}\right) , s_m \in S_m \right\rbrace.$$
Let $\delta_{KL} > 0$. For all $m \in \mathbb{N}^*$, let $\eta_m \geq 0$. There exist two functions, denoted by  $\hat{s}_{\hat{m}}$ and $\bar{s}_m$, belonging to $S_m$, such that 

$$ P_n (- \log(\hat{s}_{\hat{m}})) \leq \inf_{s_m \in S_m} P_n(-\log(s_m)) + \eta_m ;$$
\begin{equation}
\label{KLn}
KL_n(s_0,\bar{s}_m) \leq \inf_{s_m \in S_m} KL_n(s_0,s_m) + \delta_{KL}.
\end{equation}

Denote by $\hat{f}_m:=-\log\left( \frac{\hat{s}_m}{s_0} \right)$ and $\bar{f}_m:=-\log\left( \frac{\bar{s}_m}{s_0} \right)$.
Let $\eta \geq 0$ and fix $m \in \mathbb{N}^*$. We define
$$M(m)=\left\lbrace m' \in \mathbb{N}^* | P_n(-\log(\hat{s}_{m'})) + \text{pen} (m') \leq P_n(-\log(\hat{s}_m)) + \text{pen} (m) + \eta \right\rbrace.$$

\subsection{Proof of the theorem \ref{theo interm} thanks to the propositions \ref{prop1} and \ref{prop2}}
Let $M_n >0$ and $ \kappa \geq 36$. Let $C_{M_n}=\max \left(\frac{1}{a_\pi} , A_\Sigma + \frac{1}{2} (|M_n|+A_\beta)^2 A_\Sigma^2  , \frac{q(|M_n|+A_\beta)A_\Sigma}{2} \right)$.
Assume that, for all $m \in \mathbb{N}^*$, $\text{pen}(m)=\lambda m$, with
$$\lambda \geq \kappa C_{M_n} \sqrt{\frac{k}{n}} \left( 1 + q ||x||_{\text{max,n}} \log(n) \sqrt{k \log(2p+1)} \right).$$
We derive from the two propositions that there exists $\kappa '$ such that, if $\hat{m}$ satisfies
$$- \frac{1}{n} \sum_{i=1}^n \log(\hat{s}_{\hat{m}}(Y_i|X_i)) + \text{pen} (\hat{m}) \leq \inf_{m \in \mathbb{N}^*} \left( - \frac{1}{n} \sum_{i=1}^n \log(\hat{s}_{m} (Y_i|X_i))+\text{pen}(m)\right) + \eta ;$$
then $\hat{s}_{\hat{m}}$ satisfies 
\begin{align*}
E(KL_n(s_0,\hat{s}_{\hat{m}}))&=E(KL_n(s_0,\hat{s}_{\hat{m}}) \mathds{1}_T)+E(KL_n(s_0,\hat{s}_{\hat{m}}) \mathds{1}_{T^c}) \\
&\leq (1+\kappa^{-1}) \inf_{m\in\mathbb{N}^*} \left( \inf_{s_m \in S_m} KL_n (s_0,s_m)+ \text{pen}(m) + \eta_m \right) \\
&+ \kappa' \frac{C_{M_n}}{\sqrt{n}}k^{3/2}q \left( 1+(A_\beta+\tilde{A}_\Sigma)^2 \right) + \eta \\
&+ \kappa' \frac{e^{-\frac{1}{2}-\frac{1}{4} a_{\beta}^2 a_\Sigma} \pi^{q/2}}{(qA_\Sigma)^{q/2}} \sqrt{2knq a_\pi} e^{-1/4 (M_n^2 -2M_nA_\beta)a_\Sigma}.
\end{align*}

In order to optimize this equation with respect to $M_n$, we consider $M_n$ the positive solution of
$$\log(n) -\frac{1}{4} (X^2 -2X A_\beta) a_\Sigma=0 ;$$
we obtain $M_n = A_\beta + \sqrt{A_\beta^2 +\frac{4 \log(n)}{a_\Sigma}}$ and $\sqrt{n} e^{-1/4 (M_n^2 -2M_nA_\beta)a_\Sigma} = \frac{1}{\sqrt{n}}$.

On the other hand,
\begin{align*}
C_{M_n}  &\leq \left( A_\Sigma \vee \frac{1}{a_\pi} \right) \left[ 1+\frac{q+1}{2} A_\Sigma(M_n+A_\beta)^2 \right] \\
&\leq \left( A_\Sigma \vee \frac{1}{a_\pi} \right) \left[ 1+4(q+1) A_\Sigma \left(A_\beta^2+\frac{\log(n)}{a_\Sigma} \right) \right]. 
\end{align*}
We obtain
\begin{align*}
E(KL_n(s_0,\hat{s}_{\hat{m}}))&=E(KL_n(s_0,\hat{s}_{\hat{m}}) \mathds{1}_T)+E(KL_n(s_0,\hat{s}_{\hat{m}}) \mathds{1}_{T^c})  \\
&\leq (1+\kappa^{-1}) \inf_{m\in\mathbb{N}^*} \left( \inf_{s_m \in S_m} KL_n (s_0,s_m)+ \text{pen}(m) + \eta_m \right) +\eta \\
& + \kappa'\sqrt{\frac{k}{n}} \left[ \frac{e^{-\frac{1}{2}-\frac{1}{4} a_{\beta}^2 a_\Sigma} \pi^{q/2}}{(qA_\Sigma)^{q/2}} \sqrt{2q a_\pi} \right. \\
&+ \left. \left( A_\Sigma \vee \frac{1}{a_\pi} \right) \left( 1+4(q+1) A_\Sigma\left(A_\beta^2+\frac{\log(n)}{a_\Sigma} \right)\right) k(1+(A_\beta+\tilde{A}_\Sigma)^2 ) \right] .
\end{align*}

\subsection{Proof of the theorem \ref{inegOracle}}

Let $\lambda >0$. Let $\hat{m}=\inf \{ m \in \mathbb{N} | \hat{s}(\lambda) \in S_{\hat{m}} \} = \lceil | \hat{s}(\lambda) |_1 \rceil $. Then,
as $\hat{s}(\lambda)$ is the Lasso estimator, and as $S_m = \{s_\xi \in S: |s_\xi|_1 \leq m \}$, we could write
\begin{align*}
-\frac{1}{n} \sum_{i=1}^{n} \log(\hat{s}(\lambda) (Y_i|X_i)) + \lambda \hat{m} &\leq -\frac{1}{n} \sum_{i=1}^{n} \log(\hat{s}(\lambda) (Y_i|X_i)) + \lambda |\hat{s}(\lambda)|_1 + \lambda \\
 &= \inf_{s_\xi \in S_m} \left\{ -\frac{1}{n} \sum_{i=1}^{n} \log(s_\xi) (Y_i|X_i) + \lambda |s_\xi|_1 \right\} + \lambda \\
 &=\inf_{m \in \mathbb{N}^*} \inf_{s_\xi, |\xi|_1 \leq m} \left\{- \frac{1}{n} \sum_{i=1}^{n} \log(s_\xi) (Y_i|X_i) + \lambda |s_\xi|_1 \right\} + \lambda \\
 &=\inf_{m \in \mathbb{N}^*} \inf_{s_\xi, |\xi|_1 \leq m} \left\{ -\frac{1}{n} \sum_{i=1}^{n} \log(s_m) (Y_i|X_i) + \lambda m \right\} + \lambda . \\
\end{align*}
Taking $\text{pen}(m)=\lambda m$, $\eta=\lambda$, and $\hat{s}_m$ such that
$$ -\frac{1}{n} \sum_{i=1}^n  (\hat{s}_m(Y_i|X_i)) \leq \inf_{s_m \in S_m} \left( -\frac{1}{n} \sum_{i=1}^n\log(s_m (Y_i|X_i)) \right)$$
with $\eta_m=0$; then, $\hat{s}(\lambda)$ satisfies
$$ -\frac{1}{n} \sum_{i=1}^n \log(\hat{s}_{\hat{m}}(Y_i|X_i)) + \text{pen}(\hat{m}) \leq \inf_{m \in \mathbb{N}} \left( -\frac{1}{n} \sum_{i=1}^n\log(s_m (Y_i|X_i)) + \text{pen}(m) \right) + \eta.$$

All assumptions of the theorem \ref{theo interm} are satisfied, which leads to the oracle inequality.
\section{Proofs of propositions \ref{prop1} and \ref{prop2}}
\label{proofProposition}

\subsection{Proof of the proposition \ref{prop1}}

In this proposition, we will prove the main theorem according to the event $T$.
For that, we need some preliminary results.

From our notations, reminded in section \ref{notations}, we have, for all $m\in \mathbb{N}^*$ for all $m' \in M(m)$,
\begin{align}
\label{contexte}
 P_n(\hat{f}_{m'}) + \text{pen} (m') &\leq P_n(\hat{f}_m)+\text{pen}(m)+\eta \leq P_n(\bar{f}_m) + \text{pen} (m) +\eta_m +\eta ; \nonumber \\
 E_X(P_n(\hat{f}_{m'}))+\text{pen}(m') &\leq E_X(P_n(\bar{f}_m))+\text{pen}(m) +\eta_m+\eta +\nu_n(\bar{f}_m) -\nu_n(\hat{f}_{m'}) ; \nonumber \\
KL_n(s_0,\hat{s}_{m'})+\text{pen}(m') &\leq \inf_{s_m \in S_m} KL_n(s_0,s_m) +\delta_{KL} +\text{pen}(m) +  \eta_m + \eta + \nu_{n} (\bar{f}_m) - \nu_{n}(\hat{f}_{m'}); 
\end{align}

thanks to the inequality \eqref{KLn}.

The goal is to bound $-\nu_{n}(\hat{f}_{m'})= \nu_{n}(-\hat{f}_{m'})$.

To control this term, we use the following lemma.

\begin{lem}
\label{megalemme}
Let $M_n >0$. Let $$T=\left\lbrace \max_{i \in \{1,\ldots,n\}} \left( \max_{z\in \{1,\ldots,q\}} |Y_{i,z}|\right) \leq M_n \right\rbrace.$$
Let $C_{M_n}=\max \left(\frac{1}{a_\pi} , A_\Sigma + \frac{1}{2} (|M_n|+A_\beta)^2 A_\Sigma^2  , \frac{q(|M_n|+A_\beta)A_\Sigma}{2} \right)$ and 
$$\Delta_{m'} =m' ||x||_{\max,n} \log (n) \sqrt{k \log(2p+1)} + 6(1+k(A_\beta+\tilde{A}_\Sigma)).$$
Then, on the event $T$, for all $m' \in \mathbb{N}^*$, for all $t>0$, with probability greater than $1-e^{-t}$, 
$$\sup_{f_{m'} \in \mathcal{F}_{m'}} |\nu_{n}(-f_{m'})| \leq  \frac{4 C_{M_n}}{\sqrt{n}} \left( 9 \sqrt{k} q \Delta_{m^{'}}+\sqrt{2} \sqrt{t}(1+k(A_\beta +\tilde{A}_\Sigma)) \right)$$
\end{lem}
\begin{proof}
 Page \pageref{proofMegaLemme}
\end{proof}

From \eqref{contexte},
on the event $\mathcal{T}$, for all $ m \in \mathbb{N}^*$, for all $  m' \in M(m)$, for all $t >0$, with probability greater than $1-e^{-t}$,
\begin{align*}
KL_n(s_0,\hat{s}_{m'}) +  \text{pen} (m') &\leq \inf_{s_m \in S_m} KL_n(s_0,s_m) + \delta_{KL} +\text{pen} (m)+\nu_n(\bar{f}_m)\\
& + \frac{4 C_{M_n}}{\sqrt{n}} \left( 9 \sqrt{k} q \Delta_{m^{'}}+\sqrt{2} \sqrt{t}(1+k(A_\beta +\tilde{A}_\Sigma) \right)  + \eta_m + \eta\\
& \leq \inf_{s_m \in S_m} KL_n(s_0,s_m) +  \text{pen} (m) + \nu_{n}(\bar{f}_m) \\
&+ 4 \frac{C_{M_n}}{\sqrt{n}} \left( 9\sqrt{k}q\Delta_{m^{'}}+\frac{1}{2\sqrt{k}}(1+k(A_\beta+\tilde{A}_{\Sigma}))^2 +\sqrt{k} t \right)\\
&+ \eta_m + \eta + \delta_{KL},
\end{align*}

the last inequality being true because $2 a b \leq \frac{1}{\sqrt{k}} a^2 + \sqrt{k} b^2$.
Let $z>0$ such that $t=z+m+m'$. On the event $\mathcal{T}$, for all $ m \in \mathbb{N}$, for all $m' \in M(m)$, with probability greater than $1-e^{-(z+m+m')}$,
\begin{align*}
KL_n(s_0,\hat{s}_{m'}) +  \text{pen} (m')& \leq \inf_{s_m \in S_m} KL_n(s_0,s_m) +  \text{pen} (m) + \nu_{n}(\bar{f}_m)\\
&+  4 \frac{C_{M_n}}{\sqrt{n}} \left( 9\sqrt{k}q \Delta_{m^{'}} + \frac{1}{2 \sqrt{k}} (1+k(A_\beta+\tilde{A}_\Sigma))^2 +\sqrt{k}(z+m+m^{'}) \right)\\
&+ \eta_m + \eta + \delta_{KL} .
\end{align*}

\begin{align*}
KL_n(s_0,\hat{s}_{m'}) - \nu_{n}(\bar{f}_m) & \leq \inf_{s_m \in S_m} KL_n(s_0,s_m) + \text{pen}(m) + 4\frac{C_{M_n}}{\sqrt{n}}\sqrt{k} m \\
&+  \left[ \frac{4 C_{M_n}}{\sqrt{n}} \sqrt{k}(m^{'}+9 q\Delta_{m^{'}})-\text{pen}(m^{'}) \right] \\
&+ \frac{4 C_{M_n}}{\sqrt{n}} \left( \frac{1}{2 \sqrt{k}}(1+k (A_\beta+\tilde{A}_{\Sigma}))^2 +\sqrt{k}z \right) +\eta_m +\eta +\delta_{KL}
\end{align*}

Let $\kappa \geq 1$, and assume that $ \text{pen} (m)=\lambda m$ with
$$\lambda \geq \frac{4 C_{M_n}}{\sqrt{n}} \sqrt{k} \left(1+9q||x||_{max,n} \log(n) \sqrt{k \log(2p+1)} \right)$$

Then
\begin{align*}
KL_n(s_0,\hat{s}_{m'}) - \nu_{n}(\bar{f}_m)
& \leq \inf_{s_m \in S_m} KL_n(s_0,s_m) + (1+\kappa^{-1})  \text{pen} (m)  \\
&+ \frac{4 C_{M_n}}{\sqrt{n}} \left( \frac{1}{2 \sqrt{k}} (1+k(A_\beta+\tilde{A}_\Sigma))^2 +54\sqrt{k}q(1+k(A_\beta+\tilde{A}_\Sigma)) +\sqrt{k}z\right)
&+ \eta + \delta_{KL}  + \eta_m\\
&\leq  \inf_{s_m \in S_m} KL_n(s_0,s_m) + (1+\kappa^{-1})  \text{pen} (m)  \\
&+ \frac{4 C_{M_n}}{\sqrt{n}} \left( 27 k^{3/2} + \frac{1}{\sqrt{k}} (1+ k(A_\beta + \tilde{A}_\Sigma))^2 (\frac{1}{2}+27) + \sqrt{k}z \right)\\
&+ \eta_m +\eta +\delta_{KL}.
\end{align*}

Let $\hat{m}$ such that
$$- \frac{1}{n} \sum_{i=1}^n \log(\hat{s}_{\hat{m}}(Y_i|X_i)) +  \text{pen} (\hat{m}) \leq \inf_{m \in \mathbb{N}^*} \left( - \frac{1}{n} \sum_{i=1}^n \log(\hat{s}_{m} (Y_i|X_i))+ \text{pen} (m)\right) + \eta ;$$
and $M(m)=\left\lbrace m' \in \mathbb{N}^* | P_n(-\log(\hat{s}_{m'})) +  \text{pen} (m') \leq P_n(-\log(\hat{s}_m)) +  \text{pen} (m) + \eta \right\rbrace.$ By definition, $\hat{m} \in M(m)$. 
Because for all $m \in \mathbb{N}^*$, for all $m' \in M(m)$,
$$1- \sum_{m \in \mathbb{N}^*, m' \in M(m)} e^{-(z+m+m')} \geq 1-e^{-z} \sum_{m,m' \in (\mathbb{N}^{*})^2} e^{-m-m'} \geq 1-e^{-z} ,$$
we could sum up over all models.

On the event  $\mathcal{T}$, for all $z>0$, with probability greater than $1-e^{-z}$,
\begin{align*}
KL_n(s_0,\hat{s}_{\hat{m}}) - \nu_{n}(\bar{f}_m)
&\leq  \inf_{m \in \mathbb{N}^* } \left(\inf_{s_m \in S_m} KL_n(s_0,s_m) + (1+\kappa^{-1})  \text{pen} (m) +\eta_m \right) \\
& + \frac{4 C_{M_n}}{\sqrt{n}} \left( 27 k^{3/2} + \frac{55q}{2\sqrt{k}}  (1+ k(A_\beta + \tilde{A}_\Sigma))^2  + \sqrt{k}z \right)\\
&+\eta +\delta_{KL}.
\end{align*}
By integrating over $z>0$, and noticing that $E(\nu_n(\bar{f}_m))=0$ and that $\delta_{KL}$ can be chosen arbitrary small, we get
\begin{align*}
 E(KL_n(s_0,\hat{s}_{\hat{m}}) \mathds{1}_T) &\leq \inf_{m\in \mathbb{N}^{*}} \left( \inf_{s_m \in S_m} KL_n(s_0,s_m) + (1+\kappa^{-1}) \text{pen}(m) +\eta_m \right)\\
 & +\frac{4 C_{M_n}}{\sqrt{n}} \left(27 k^{\frac{3}{2}} + \frac{q}{\sqrt{k}} \frac{55}{2}(1+k(A_\beta+\tilde{A}_\Sigma))^2 +\sqrt{k}\right) +\eta\\
 &\leq \inf_{m\in \mathbb{N}^{*}} \left( \inf_{s_m \in S_m} KL_n(s_0,s_m) + (1+\kappa^{-1}) \text{pen}(m) +\eta_m \right)\\
 &+\frac{332 k^{\frac{3}{2}}q C_{M_n}}{\sqrt{n}} (1+(A_\beta+\tilde{A}_\Sigma)^2)+\eta.
\end{align*}

\subsection{Proof of the proposition \ref{prop2}}
We want an upper bound of $E\left(KL_n(s_0,\hat{s}_{\hat{m}}) \mathds{1}_{T^c} \right)$.
Thanks to the Cauchy Schwarz inequality,
$$E\left(KL_n(s_0,\hat{s}_{\hat{m}}) \mathds{1}_{T^c} \right)  \leq \sqrt{ E(KL_n^2(s_0,\hat{s}_{\hat{m}}))} \sqrt{P(T^c)}.$$
However, 
\begin{align*}
KL_n(s_0,s_\xi) &= \int_{\mathbb{R}^q} \log\left(\frac{s_0(y|x)}{s_\xi(y|x)}\right) s_o(y|x) dy \\
&= \int_{\mathbb{R}^q} \log(s_0(y|x)) s_0(y|x) dy - \int_{\mathbb{R}^q} \log(s_\xi(y|x)) s_0(y|x) dy \\
&\leq -\int_{\mathbb{R}^q} \log(s_\xi(y|x)) s_0(y|x) dy.
\end{align*}
Because parameters are assumed to be bounded, according to the assumption \eqref{paraborn}, we get, with $(\beta^0, \Sigma^0, \pi^0)$ the parameters of $s_0$ and $(\beta,\Sigma,\pi)$ the parameters of $s_\xi$,
\begin{align*}
\log(s_\xi(y|x)) s_0(y|x) &= \log \left( \sum_{r=1}^k \frac{\pi_r}{(2\pi)^{q/2} \sqrt{\det(\Sigma_r)}} \exp \left(-\frac{(y-\beta_r x)^t \Sigma_r^{-1} (y-\beta_r x)}{2} \right)\right) \\
&\times \sum_{r=1}^k \frac{\pi_r^0}{(2\pi)^{q/2} \sqrt{\det(\Sigma_r^0)}} \exp \left(-\frac{(y-\beta_r^0 x)^t \Sigma_r^{0,-1} (y-\beta_r^0 x)}{2} \right)\\
&\geq \log\left( k \frac{a_\pi \sqrt{\det(\Sigma_r^{-1})}}{(2 \pi)^{q/2} } \exp \left(- (y^t \Sigma_r^{-1} y + x^t \beta_r^t \Sigma_r^{-1}\beta_r x) \right) \right)\\
&\times k \frac{a_\pi \sqrt{\det(\Sigma_r^{0,-1})}}{(2 \pi)^{q/2}} \exp \left(- (y^t \Sigma_r^{-1} y + x^t \beta_r^t \Sigma_r^{-1}\beta_r x) \right)\\
&\geq \log\left( k \frac{a_\pi a_\Sigma^{q/2}}{(2 \pi)^{q/2} } \exp \left(- q(y^t y  + A_\beta^2) A_\Sigma \right) \right)\\
&\times k \frac{a_\pi a_\Sigma^{q/2}}{(2 \pi)^{q/2} } \exp \left(- q(y^t y  + A_\beta^2) A_\Sigma \right).
\end{align*}
Indeed,
\begin{align*}
 |z^t \Sigma z| &\leq \sum_{z_1=1}^q \sum_{z_2=1}^q |z_{z_1} \Sigma_{z_1,z_2} z_{z_2}|\\
 &\leq \max_{z_1,z_2} |\Sigma_{z_1,z_2}| ||z||_1^2 \leq q \max_{z_1,z_2} |\Sigma_{z_1,z_2}| ||z||_2^2
\end{align*}

To recognize the expectation of a Gaussian standardized variables, we put $u= \sqrt{2 q A_\Sigma} y$:
\begin{align*}
KL(s_0,s_\xi) &\leq - \frac{k a_\pi e^{-q A_\beta^2 A_\Sigma} a_\Sigma^{q/2} }{(2 q A_\Sigma)^{q/2} } \int_{\mathbb{R}^q} \left[ \log\left(\frac{k a_\Sigma^{q/2}  a_\pi}{(2 \pi)^{q/2}} \right)- q A_\beta^2 A_\Sigma - \frac{u^t u}{2} \right]\frac{e^{\frac{- u^t u}{2}}}{(2 \pi)^{q/2}} du \\
& \leq - \frac{a_\Sigma ^{q/2} k a_\pi e^{-A_\beta^2 A_\Sigma q}}{(2 q A_\Sigma)^{q/2}} E\left[ \log\left(\frac{k a_\pi a_\Sigma^{q/2}}{(2 \pi)^{q/2}} \right)- q A_\beta^2 A_\Sigma - \frac{U^2}{2} \right] \\
& \leq  - \frac{k a_\Sigma^{q/2} a_\pi e^{-A_\beta^2 A_\Sigma q}}{(2 q A_\Sigma)^{q/2}} \left[ \log\left(\frac{k a_\pi a_\Sigma^{q/2}}{(2 \pi)^{q/2}} \right)- q A_\beta^2 A_\Sigma - \frac{1}{2} \right]\\
& \leq - \frac{k a_\Sigma^{q/2} a_\pi e^{-A_\beta^2 A_\Sigma q- 1/2}}{(2\pi)^{q/2} (qA_\Sigma)^{q/2}} e^{1/2} \pi^{q/2} \log\left( \frac{k a_\pi e^{-q A_\beta^2 A_\Sigma -1/2} a_\Sigma^{q/2}}{(2\pi)^{q/2} }\right)\\
&\leq \frac{e^{-1/2} \pi^{q/2}}{ (q A_\Sigma)^{q/2}}
\end{align*}
where $U \sim \mathcal{N}_q (0,1)$. We have used that for all $t \in \mathbb{R}$, $t \log(t) \geq - e^{-1}$.
Then, we get
$$KL_n(s_0,s_\xi) \leq \frac{1}{n} \sum_{i=1}^n KL(s_0(.|x_i),s_\xi(.|x_i)) \leq \frac{e^{-1/2} \pi^{q/2}}{ (q A_\Sigma)^{q/2}} ;$$
and
$$ \sqrt{E(KL_n^2(s_0,\hat{s}_{\hat{m}}))} \leq \frac{e^{-1/2} \pi^{q/2}}{ (q A_\Sigma)^{q/2}}.$$

For the last step, we need to bound $P(T^c)$.

$$P(T^c) = E(\mathds{1}_{T^c})= E(E_X(\mathds{1}_{T^c}))=E(P_X(T^c)) \leq E\left(\sum_{i=1}^n P_X(||Y_i||_\infty > M_n) \right).$$
Nevertheless, $Y_i|X_i \sim \sum_{r=1}^k \pi_{r} \mathcal{N}_q (\beta_r X_i, \Sigma_r)$, then,

\begin{align*}
P(||Y||_\infty > M_n) &= \int_{\mathbb{R}^q} \mathds{1}_{\{||Y||_{\infty} \geq M_n\}} \sum_{r=1}^k \pi_r \frac{1}{(2\pi)^{q/2} \sqrt{\det(\Sigma_r)}} \exp \left( -\frac{(y-\beta_{r} x_i)^t \Sigma_{r}^{-1}(y-\beta_{r} x_i)}{2} \right) dy \\
&= \sum_{r=1}^k \pi_r \int_{\mathbb{R}^q} \mathds{1}_{\{||Y||_\infty \geq M_n\}}  \frac{1}{(2\pi)^{q/2} \sqrt{\det(\Sigma_r)}} \exp \left( -\frac{(y-\beta_{r} x_i)^t \Sigma_{r}^{-1}(y-\beta_{r} x_i)}{2} \right) dy \\
&=\sum_{r=1}^k \pi_r P_X(||Y_r||_{\infty} > M_n) \leq \sum_{r=1}^k \sum_{z=1}^q \pi_r P_X (|Y_{r,z}| > M_n).\\
\end{align*}
with $Y_r \sim N(\beta_r X_i, \Sigma_r)$ and $Y_{r,z} \sim N(\beta_{r,z}x,\Sigma_{r,z,z})$.

We need to control $P_X(|Y_{r,z}|>M_n)$, for all $z \in \{1,\ldots,q\}$.

\begin{align*}
  P_X (|Y_{r,z}| > M_n)&=  P_X (Y_{r,z} > M_n) +P_X (Y_{r,z} < - M_n) \\
  &= P_X \left(U > \frac{M_n- \beta_{r,z}x}{\sqrt{\Sigma_{r,z,z}}}\right)+ P_X \left(U <  \frac{-M_n- \beta_{r,z}x}{\sqrt{\Sigma_{r,z,z}}}\right) \\
&= P_X \left(U > \frac{M_n- \beta_{r,z}x}{\sqrt{\Sigma_{r,z,z}}}\right)+ P_X \left(U >  \frac{M_n + \beta_{r,z}x}{\sqrt{\Sigma_{r,z,z}}}\right) \\
&\leq e^{-\frac{1}{2} \left(\frac{M_n - \beta_{r,z}x}{\sqrt{\Sigma_{r,z,z}}} \right) ^2} + e^{-\frac{1}{2} \left(\frac{M_n + \beta_{r,z}x}{\sqrt{\Sigma_{r,z,z}}} \right) ^2}\\
&\leq 2 e^{-\frac{1}{2} \left(\frac{M_n - |\beta_{r,z}x|}{\sqrt{\Sigma_{r,z,z}}} \right) ^2}\\
&\leq 2 e^{-\frac{1}{2}\frac{M_n^2 -2 M_n |\beta_{r,z}x| + |\beta_{r,z}x|^2}{\Sigma_{r,z,z}}}.\\
\end{align*}
where $U \sim N(0,1)$.
Then, 
$$P(||Y||_\infty > M_n) = 2 k q e^{-\frac{1}{2} (M_n^2 -2 M_n A_\beta +a_\beta^2) a_\Sigma},$$
and we get $P(T^c)\leq E \left( \sum_{i=1}^n 2 k q a_\pi e^{-\frac{1}{2}(M_n^2 -2 M_n A_\beta+ a_\beta^2) a_\Sigma} \right) \leq 2 k n a_\pi q e^{-\frac{1}{2} (M_n^2 -2 M_n A_\beta +a_\beta^2) a_\Sigma} $.
We have  obtained the wanted bound for $E(KL_n(s_0, \hat{s}_{\hat{m}}) \mathds{1}_{T^c})$.

\section{Some details}

\label{proofLemme}
\subsection{Proof of the lemma \ref{megalemme}}
First, give some tools to prove the lemma \ref{megalemme}.

We define $||g||_n = \sqrt{\frac{1}{n} \sum_{i=1}^n g^2(Y_i|x_i)}$ for any measurable function $g$.

Let $m \in \mathbb{N}^*$. We have
$$\sup_{f_m \in F_m} |\nu_n (-f_m)| = \sup_{f_m \in F_m} \left| \frac{1}{n} \sum_{i=1}^n (f_m(Y_i|x_i) - E(f_m(Y_i|x_i))\right|.$$
To control the deviation of such a quantity, we shall combine concentration with symmetrization arguments.
We shall first use the following concentration inequality which can be found in \cite{BoucheronLugosiMassart}.
\begin{lem}
\label{Boucheron}
 Let $Z_1,\ldots,Z_n$ be independent random variables with values in some space $\mathcal{Z}$ and let $\Gamma$ be a class of real-valued functions on $\mathcal{Z}$. Assume that there exists $R_n$ a non-random constant such that $\sup_{\gamma \in \Gamma} ||\gamma||_n \leq R_n$. Then, for all $t>0$,
 $$ P \left(\sup_{\gamma \in \Gamma} \left| \frac{1}{n} \sum_{i=1}^n \gamma(Z_i) - E(\gamma(Z_i)) \right| > E \left[ \sup_{\gamma \in \Gamma} \left| \frac{1}{n} \sum_{i=1}^n \gamma(Z_i)-E(\gamma(Z_i)) \right| \right] + 2 \sqrt{2} R_n \sqrt{\frac{t}{n}} \right) \leq e^{-t}.$$
\end{lem}
\begin{dem}
 See \cite{BoucheronLugosiMassart}.
\end{dem}

Then, we propose to bound $E \left[ \sup_{\gamma \in \Gamma} \left| \frac{1}{n} \sum_{i=1}^n \gamma(Z_i)-E(\gamma(Z_i)) \right| \right]$ thanks to the following symmetrization argument. The proof of this result can be found in \cite{vanDerVaart}.
\begin{lem}
 Let $Z_1,\ldots,Z_n$ be independent random variables with values in some space $\mathcal{Z}$ and let $\Gamma$ be a class of real-valued functions on $\mathcal{Z}$. Let $(\epsilon_1,\ldots,\epsilon_n)$ be a Rademacher sequence independent of $(Z_1,\ldots,Z_n)$.
 Then,
 $$E \left[ \sup_{\gamma \in \Gamma} \left| \frac{1}{n} \sum_{i=1}^n \gamma(Z_i)-E(\gamma(Z_i)) \right| \right]\leq 2 E \left[ \sup_{\gamma \in \Gamma} \left| \frac{1}{n} \sum_{i=1}^n \epsilon_i \gamma(Z_i) \right| \right].$$
\end{lem}
\begin{dem}
See  \cite{vanDerVaart}. 
\end{dem}

Then, we have to control $E(\sup_{\gamma \in \Gamma} \left|\frac{1}{n} \sum_{i=1}^n \epsilon_i \gamma(Z_i)\right|)$.

\begin{lem}
\label{dernierLemme}
 Let $(Z_1,\ldots,Z_n)$ be independent random variables with values in some space $\mathcal{Z}$ and let $\Gamma$ be a class of real-valued functions on $\mathcal{Z}$. Let $(\epsilon_1,\ldots,\epsilon_n)$ be 
 a Rademacher sequence independent of $(Z_1,\ldots,Z_n)$. Define $R_n$ a non-random constant such that
 $$\sup_{\gamma \in \Gamma} ||\gamma||_n \leq R_n.$$
 Then, for all $S \in \mathbb{N}^*$,
 $$E\left[\sup_{\gamma \in \Gamma} \left|\frac{1}{n} \sum_{i=1}^n \epsilon_i \gamma(Z_i) \right| \right] \leq R_n \left( \frac{6}{\sqrt{n}} \sum_{s=1}^S 2^{-s} (\sqrt{\log(1+N(2^{-s} R_n, \Gamma,||.||_n))} + 2^{-S} \right)$$
 where $N(\delta,\Gamma,||.||_n)$ stands for the $\delta$-packing number of the set of functions $\Gamma$ equipped with the metric induced by the norm $||.||_n$.
 \end{lem}
\begin{proof}
 See \cite{Massart}.
\end{proof}

In our case, we get the following lemma

\begin{lem}
Let $m \in \mathbb{N}^*$. Consider $(\epsilon_1,\ldots, \epsilon_n)$ a Rademacher sequence independent of $(Y_1,\ldots,Y_n)$. Then, on the event $T$,
$$E \left( \sup_{f_m \in F_m} \left| \sum_{i=1}^n \epsilon_i f_m(Y_i|x_i) \right| \right) \leq 18 \sqrt{k} \frac{C_{M_n}q}{\sqrt{n}} \Delta_m$$
where $\Delta_m:=  ||x||_{\max,n} m \log(n) \sqrt{k \log(2p+1)} + 6(1+k(A_\beta + \tilde{A}_\Sigma))$.
\end{lem}
\begin{dem}
Let $m \in \mathbb{N}^*$.
Thanks to lemma \ref{lemmefastoche}, we get that on the event $\mathcal{T}$, $\sum_{f_m \in F_m} ||f_m||_n \leq R_n:= 2 C_{M_n}(1+k(A_\beta + \tilde{A}_{\Sigma}))$. Besides, on the event $\mathcal{T}$, for all $S \in \mathbb{N}^*$,
\begin{align*}
&\sum_{s=1}^S 2^{-s} \sqrt{\log[1+N(2^{-s} R_n,F_m,||.||_{n})]} \leq \sum_{s=1}^S 2^{-s} \sqrt{\log(2N(2^{-s} R_n,F_m,||.||_n))}\\
&\leq \sum_{s=1}^S 2^{-s} \left[\sqrt{\log(2)} + \sqrt{\log(2p+1)} \frac{2^{s+1} C_{M_{n}} q k m ||x||_{\max,n}}{R_n} \right. \\
&+ \left.  \sqrt{k \log  \left(1+\frac{2^{s+3} C_{M_n} q^2 k A_\Sigma}{R_n}\right) \left( 1+ \frac{2^{s+3} C_{M_n}}{R_n}\right)} \right]  \text{ thanks to lemma \ref{entropie} }\\
& \leq \sum_{s=1}^S 2^{-s} \left[\sqrt{\log(2)} + \sqrt{\log(2p+1)} \frac{2^{s+1} C_{M_{n}} q k m ||x||_{\max,n}}{R_n} \right. \\
& \left. +  \sqrt{k \log  \left( 1+2^{s+3}  \frac{C_{M_n}}{R_n} \max(1,q^2 k A_\Sigma) \right)^2	} \right] \\
&\leq \sum_{s=1}^S 2^{-s} \left[\sqrt{\log(2)} + \sqrt{\log(2p+1)} \frac{2^{s+1} C_{M_{n}}  q k m ||x||_{\max,n}}{R_n} +  \sqrt{ 2 (s+3) k \log (2) q^2 } \right]\\
&\leq \frac{2 C_{M_{n}} k m q ||x||_{\max,n}}{R_n} S \sqrt{\log(2p+1)} + \sqrt{\log(2)} \left( 1+q(\sqrt{6k} + 2 \sum_{s=1}^S 2^{-s} \sqrt{s}) \right)\\
&\leq \frac{2 C_{M_{n}} k m q ||x||_{\max,n}}{R_n} S \sqrt{\log(2p+1)} + \sqrt{\log(2)} \left( 1+q\sqrt{6k} + q\sqrt{k} \frac{\sqrt{2e}}{2 - \sqrt{e}} \right)\\
\end{align*}
because $2^{-s} \sqrt{s} \leq \left(\frac{\sqrt{e}}{2}\right)^s$ for all $s \in \mathbb{N}^*$.
Then, thanks to the lemma \ref{dernierLemme},
\begin{align*}
E\left( \sup_{f_m \in F_m} \left| \frac{1}{n} \sum_{i=1}^n \epsilon_i f_m(Y_i|x_i) \right| \right) &\leq R_n \left( \frac{6}{\sqrt{n}} \sum_{s=1}^S 2^{-s} \sqrt{\log[1+N(2^{-s} R_n,F_m,||.||_{n})]} + 2^{-S} \right)\\
&\leq R_n \left[\frac{6}{\sqrt{n}} \left( \frac{2 C_{M_n} k m q ||x||_{\max,n}}{R_n} S \sqrt{\log(2p+1)} \right. \right.\\
& \left. \left. + \sqrt{\log(2)} \left(1+q\sqrt{6k} + q\sqrt{k}\frac{2e}{2-\sqrt{e}} \right)\right) +2^{-S} \right].\\
\end{align*}
Taking $S=\frac{\log(n)}{\log(2)}$ to obtain the same order in the both terms depending on $S$, we could deduce that
\begin{align*}
&E \left( \sup_{f_m \in F_m} \left| \frac{1}{n} \sum_{i=1}^n \epsilon_i f_m(Y_i|x_i) \right| \right) \\
&\leq \frac{12 C_{M_n} k m q ||x||_{\max,n}}{\sqrt{n}} \sqrt{\log(2p+1)} \frac{\log(n)}{\log(2)} \\
&+ 2 C_{M_n} (1+k(A_\beta+\tilde{A}_\Sigma)) \left[ \frac{\sqrt{\log(2)}}{\sqrt{n}}\left(1+\sqrt{6k}+\frac{\sqrt{2e}}{2-\sqrt{2e}} \right) + \frac{1}{n} \right]\\
&\leq \frac{18 C_{M_n} k m q ||x||_{\max,n}}{\sqrt{n}} \sqrt{\log(2p+1)} \log(n)\\
& + 2 \frac{\sqrt{k}}{\sqrt{n}} C_{M_n} (1+k(A_\beta+\tilde{A}_\Sigma)) \left[\sqrt{\log(2)}\left(1+\sqrt{6}+\frac{\sqrt{2e}}{2-\sqrt{2e}} \right) + 1 \right]\\
&\leq 18 \frac{\sqrt{k}}{\sqrt{n}}C_{M_n} \left[ m q \sqrt{k \log(2p+1)} ||x||_{\max,n} \log(n) + 6 (1+k(A_\beta+\tilde{A}_\Sigma)) \right]
\end{align*}
This completes the proof.
\end{dem}

We are now able to prove the lemma \ref{megalemme}.
\begin{align*}
 \sup_{f_m \in F_m} |\nu_n(-f_m)| &= \sup_{f_m \in F_m} \left| \frac{1}{n} \sum_{i=1}^n(f_m(Y_i|X_i) - E_X (f_m(Y_i|X_i))) \right| \\
 &\leq E \left(\sup_{f_m \in F_m} \left|\sum_{i=1}^n f_m(Y_i|X_i) -E(f_m(Y_i|X_i)) \right| \right) + 2 \sqrt{2} R_n \sqrt{\frac{t}{n}} \\
 & \text{ with probability greater than } 1-e^{-t} \text{ and where } R_n \\
 & \text{ is a constant, upper bound for } ||f_m||_{\text{max},n}\\
 & \text{ and computed from the lemma \ref{lemmefastoche} } \\
 &\leq 2 E \left(\sup_{f_m \in F_m} \left| \sum_{i=1}^n \epsilon_i f_m(Y_i|X_i) \right| \right) + 2 \sqrt{2} R_n \sqrt{\frac{t}{n}} \\
 & \text{ with } \epsilon_i \text{ a Rademacher sequence, }\\
 & \text{ independent of } Z_i \\
 &\leq 2\left(18 \sqrt{k} \frac{C_{M_n}q}{\sqrt{n}} \Delta_m\right) + 2 \sqrt{2} R_n \sqrt{\frac{t}{n}} \\
 &\leq 4 C_{M_n} \left( 9 \frac{\sqrt{k}q}{\sqrt{n}}\Delta_m + \sqrt{2} \sqrt{\frac{t}{n}}(1+ k(A_\beta + \tilde{A}_\Sigma)) \right).
\end{align*}

\subsection{Lemma \ref{lemmefastoche} and Lemma \ref{entropie}}
\begin{lem}
\label{lemmefastoche}
 On the event $\mathcal{T}=\{\max_{i\in\{1,\ldots,n\}} \max_{z\in \{1,\ldots,q\}} |Y_{i,z}|\leq M_n\}$, for all $m \in \mathbb{N}^*$,
$$\sup_{f_m \in F_m} ||f_m||_n \leq 2 C_{M_n} (1+k(A_\beta + \tilde{A}_\Sigma)):=R_n.$$
\end{lem}
\begin{dem}
\label{proofMegaLemme}
 Let $m\in \mathbb{N}^*$. Because $ f_m \in \mathcal{F}_m=\{f_m=-\log\left(\frac{s_m}{s_0}\right), s_m \in S_m\}$, there exists $s_m \in S_m$ such that $f_m=- \log\left( \frac{s_m}{s_0} \right)$.
For all $x \in \mathbb{R}^p$, denote $\xi(x) =(\beta _r x, \Sigma_r,\pi_r)_{r=1,\ldots,k}$ the parameters of $s_m(.|x)$. For all $i=1,\ldots,n$,
\begin{align*}
|f_m(Y_i|X_i)|\mathds{1}_T &= |\log(s_m(Y_i|X_i))-\log(s_0(Y_i|X_i))| \mathds{1}_T \\
& \leq \sup_{x \in \mathbb{R}^p} \sup_{\xi} \left|\frac{\partial \log (s_\xi(Y_i|x))}{\partial \xi} \right| ||\xi(x_i) - \xi_0(x_i)||_1 \mathds{1}_T,
\end{align*}
thanks to the Taylor formula.
Then, we need an upper bound of the partial derivate.

Let $s_\xi \in S_m$, with $\xi = (\beta_r,\Sigma_r,\pi_r)_{r=1,\ldots,k}$. For all $x \in \mathbb{R}^p$, for all $y \in \mathbb{R}^m$,
$$\log(s_{\xi}(y|x)) = \log \left( \sum_{r=1}^k f_r (x,y) \right)$$
where
$$f_r(x,y)=\frac{\pi_r}{(2 \pi)^{q/2} \det \Sigma_{r}} \exp\left[ -\frac{1}{2} \left(  \sum_{z_2=1}^q \left( \sum_{z_1=1}^q (y_{z_1} - \sum_{j=1}^p x_j \beta_{j,r,z_1}\right) \Sigma_{r,z_1,z_2}^{-1} \right) \left( y_{z_2} - \sum_{j=1}^p \beta_{r,j,z_2}x_j \right) \right].$$
Then,
\begin{align*}
\left| \frac{\partial\log(s_\xi(y|x))}{\partial(\beta_{l,z_1} x)} \right| &= \left| \frac{f_l(x,y)}{\sum_{r=1}^k f_r(x,y)} \right| \left( -\frac{1}{2} \sum_{z_2=1}^q \Sigma_{r,z_1,z_2}^{-1} (\beta_{r,z_2} x -y_{z_{2}}) \right) \leq \frac{q(|y|+A_\beta)A_\Sigma}{2} ;\\
\left| \frac{\partial\log(s_\xi(y|x))}{\partial(\Sigma_{l,z_1,z_2} )} \right| &= \frac{1}{\sum_{r=1}^k f_r(x,y)} \left| \frac{-f_l \text{Cof}_{z_1,z_2}(\Sigma_r)}{\det({\Sigma_r})} -\frac{f_l(x,y) (y_{z_1} -\beta_{r,z_1}x)(y_{z_2} -\beta_{r,z_2}x)\Sigma_{r,z_1,z_2}^{-2}}{2}\right|\\
&\leq \left| \frac{- \text{Cof}_{z_1,z_2}(\Sigma_r)}{\det({\Sigma_r})} +\frac{(y_{z_1} -\beta_{r,z_1}x)(y_{z_2} -\beta_{r,z_2}x) \Sigma_{r,z_1,z_2}^{-2}}{2}\right| \\
&\leq A_\Sigma + \frac{1}{2} (|y|+A_\beta)^2 A_\Sigma^2 ,
\end{align*}
where $\text{Cof}_{z_1,z_2}(\Sigma_r)$ is the $(z_1,z_2)$-cofactor of $\Sigma_r$.
We also have
$$\left| \frac{\partial \log(s_\xi (y,x))}{\partial \pi_l} \right| = \left| \frac{f_l(x,y)}{\pi_{l} \sum_{r=1}^k f_r (x,y)} \right| \leq \frac{1}{a_\pi}.$$
Thus, for all $y \in \mathbb{R}^q$,
$$\sup_{x \in \mathbb{R}^p} \sup_{\xi \in \tilde{\Xi}} \left|\frac{\partial \log (s_\xi(y|x))}{\partial \xi} \right| \leq  \max \left(\frac{1}{a_\pi} , A_\Sigma + \frac{1}{2} (|y|+A_\beta)^2 A_\Sigma^2  , \frac{q(|y|+A_\beta)A_\Sigma}{2} \right) = C_y.$$
We have $C_y\leq \left( A_\Sigma \wedge \frac{1}{a_\pi} \right) \left[ 1+ \frac{q+1}{2} A_\Sigma (|y|+A_\beta)^2 \right]$.

\begin{align*}
|f_m(Y_i|X_i)|\mathds{1}_T &\leq C_{Y_i} ||\xi(x_i) - \xi_0(x_i)||_1 \mathds{1}_T \\
&\leq C_{M_n} \sum_{r=1}^k (|\beta_r x_i -\beta_r^0 x_i| + |\Sigma_{r} - \Sigma_r^0 | + |\pi_r -\pi_r^0|).
\end{align*}
Since $f_m$ and $f^0_m$ belong to $\tilde{\Xi}$, we obtain
$$|f_m(Y_i|X_i)|\mathds{1}_T \leq 2 C_{M_n} (k A_\beta + k \tilde{A}_\Sigma +1)$$
and $||f_m||_n  \mathds{1}_T \leq 2 C_{M_n} (k A_\beta + k \tilde{A}_\Sigma +1)$ and

$$\sup_{f_m \in F_m} ||f_m||_n \mathds{1}_T \leq 2 C_{M_n} (k A_\beta + k \tilde{A}_\Sigma +1).$$
\end{dem}

For the next results, we need the following lemma, proved in \cite{Meynet}.
\begin{lem}
\label{lem7.3}
Let $\delta >0$ and $(X_{i,j})_{i=1,\ldots,n , j=1,\ldots,p} \in \mathbb{R}^{n \times p}$. There exists a family $B$ of $(2p+1)^{||X||^2_{\text{max},n} /\delta^2}$ vectors of $\mathbb{R}^p$ such that for all $\beta \in \mathbb{R}^p$ in the $\ell_1$-ball, there exists $\beta ' \in B$ such that
$$\frac{1}{n} \sum_{i=1}^n \left( \sum_{j=1}^p ||\beta_j - \beta_j '|| X_{i,j} \right)^2 \leq \delta^2.$$
\end{lem}
\begin{dem}
 See \cite{Meynet}.
\end{dem}

With this lemma, we can prove the following one:
\begin{lem}
\label{entropie}
Let $\delta >0$ and $m \in \mathbb{N}^*$. On the event $T$, we have the upper bound of the $\delta$-packing number of the set of functions $F_m$ equipped with the metric induced by the norm $||.||_n$:
$$N(\delta,F_m,||.||_n)\leq (2p+1)^{4 C_{M_n}^2 k^2 q^2 m^2 ||X||^2_{\max,n}/\delta^2} \left( 1+ \frac{8 C_{M_n} q^2k A_\Sigma}{\delta}\right) ^k \left(1+\frac{8 C_{M_n}}{\delta}\right) ^k.$$
\end{lem}
\begin{dem}
Let $m \in \mathbb{N}^*$ and $f_m \in F_m$. There exists $s_m \in S_m$ such that $f_m = - \log(s_m /s_0)$. Introduce $s'_m$ in $S$ and put $f'_m=- \log(s'_m /s_0)$.
Denote by $(\beta_r,\Sigma_r,\pi_r)_{r=1,\ldots,k}$ and $(\beta'_r,\Sigma'_r,\pi'_r)_{r=1,\ldots,k}$ the parameters of the densities $s_m$ and $s'_m$ respectively.
First, applying Taylor's inequality, on the event $\mathcal{T}=\{\max_{i\in \{1,\ldots,n\}} \max_{z\in \{1,\ldots,q\}} |Y_i|\leq M_n\}$, we get, for all $i=1,\ldots,n$,
\begin{align*}
 |f_m(Y_i|X_i)-f'_m(Y_i|X_i)| \mathds{1}_{\mathcal{T}} &= |\log(s_m(Y_i|X_i))-\log(s'_m(Y_i|X_i))| \mathds{1}_{\mathcal{T}} \\
 &\leq \sup_{x \in \mathbb{R}^p} \sup_{\xi \in \tilde{\Xi}} \left| \frac{\partial \log(s_\xi(Y_i|x_i))}{\partial{\xi}} \right| ||\xi(x_i) -\xi'(x_i)||_1 \mathds{1}_{\mathcal{T}}\\
 &\leq C_{M_n} \sum_{r=1}^k \left( \sum_{z=1}^q\left|  \beta_{r,z} x_i - \beta'_{r,z} x_i \right| + ||\Sigma_{r} - \Sigma'_r ||_1 + |\pi_r -\pi'_r|\right).
\end{align*}
Thanks to the Cauchy-Schwarz inequality, we get that
\begin{align*}
 &(f_m(Y_i|X_i) -f'_m(Y_i|X_i))^2 \mathds{1}_{\mathcal{T}} \leq 2 C_{M_n}^2 \left[\left(\sum_{r=1}^k \sum_{z=1}^q \left| \beta_r X_i - \beta'_r X_i \right| \right)^2 + (||\Sigma - \Sigma' ||_1 + ||\pi -\pi'||)^2 \right] \\
 &\leq 2 C_{M_n}^2 \left[ k q \sum_{r=1}^k \sum_{z=1}^q \left( \sum_{j=1}^p \beta_{r,j,z} X_{i,j} - \sum_{j=1}^p \beta'_{r,j,z} X_{i,j} \right)^2 + (||\Sigma - \Sigma' ||_1 + ||\pi -\pi'||)^2 \right],
\end{align*}
and
\begin{align*}
||f_m-f'_m||^2_n \mathds{1}_{\mathcal{T}} \leq &2 C_{M_n}^2 \left[ k q \sum_{r=1}^k \sum_{z=1}^q \frac{1}{n} \sum_{i=1}^n \left( \sum_{j=1}^p \beta_{r,j,z} X_{i,j} - \sum_{j=1}^p \beta'_{r,j,z} X_{i,j} \right)^2  \right. \\
& \left. \phantom{\left(\frac{1}{1}\right)^2}+ (||\Sigma - \Sigma' ||_1 + ||\pi -\pi'||)^2 \right].
\end{align*}
Denote by $a= k q \sum_{r=1}^k \sum_{z=1}^q \frac{1}{n} \sum_{i=1}^n \left( \sum_{j=1}^p \beta_{r,j,z} X_{i,j} - \sum_{j=1}^p \beta'_{r,j,z} X_{i,j} \right)^2$.
Then, for all $\delta >0$, if $a \leq \delta^2/(4C_{M_n}^2)$, $||\Sigma - \Sigma' ||_1 \leq \delta/(4 C_{M_n}) $ 
and $ ||\pi -\pi'|| \leq \delta/(4 C_{M_n})$, then $||f_m-f'_m||^2_n \leq \delta^2$.
To bound $a$, we write
$$a = k q m^2 \sum_{r=1}^k \sum_{z=1}^q \frac{1}{n} \sum_{i=1}^n \left( \sum_{j=1}^p \frac{\beta_{r,j,z}}{m} X_{i,j} - \sum_{j=1}^p \frac{\beta'_{r,j,z}}{m} X_{i;j} \right)^2$$
and we apply lemma \ref{lem7.3} to $\beta_{r,.,z}/m$ for all $r\in \{1,\ldots,k\}$, and for all $z \in \{1,\ldots,q\}$. Since $s_m \in S_m$, we have $\sum_{z=1}^q \sum_{j=1}^p \left|\frac{\beta_{r,j,z}}{m} \right| \leq1$
and thus there exists a family $\mathcal{B}$ of $(2p+1)^{4 C_{M_n}^2 q^2 k^2 m^2 ||x||^2_{\max,n}/\delta^2}$ vectors of $\mathbb{R}^p$ such that for all $r\in \{1,\ldots, k\}$, for all $z \in \{1,\ldots,q\}$, for all $\beta_{r,.,z}$, there exists
$\beta'_r \in \mathcal{B}$ such that $a\leq \delta^2/(4 C_{M_n}^2)$. 
Moreover, since $||\Sigma||_1 \leq q^2 k A_\Sigma$ and $||\pi||_1\leq 1$, we get that, on the event $\mathcal{T}$,
\begin{align*}
N(\delta,F_m,||.||_n)& \leq \text{card}(\mathcal{B}) N\left(\frac{\delta}{4 C_{M_n}},B_1^k(q^2 k A_\Sigma),||.||_1\right) N\left(\frac{\delta}{4 C_{M_n}}, B_1^k(1),||.||_1 \right)\\
&\leq (2p+1)^{4 C_{M_n}^2 q^2 k^2 m^2 ||x||^2_{\max,n}/\delta^2} \left( 1+ \frac{8 C_{M_n} q^2 k A_\Sigma}{\delta}\right) ^k \left(1+\frac{8 C_{M_n}}{\delta}\right) ^k
\end{align*}
\end{dem}

\section{Acknowledgment}
 I am grateful to Pascal Massart for suggesting me to study this problem, and for stimulating discussions.

\bibliography{biblio}
\bibliographystyle{plain}
\end{document}